\documentclass[12pt]{article}
\title{Diophantine quintuples containing triples of the first kind}
\author{D.J. Platt\\ Heilbronn Institute for Mathematical Research \\ University of Bristol, Bristol, UK\\ dave.platt@bris.ac.uk\\ \\
and\\ \\
T.S. Trudgian\footnote{Supported by Australian Research Council DECRA Grant DE120100173.}\\
Mathematical Sciences Institute\\ The Australian National University,
 ACT 0200, Australia\\ timothy.trudgian@anu.edu.au
}
\usepackage{url}
\usepackage{amsthm}
\usepackage{amsmath}
\usepackage{fullpage}
\usepackage{enumerate}
\usepackage{amssymb}
\usepackage{booktabs}
\usepackage[boxed]{algorithm2e}
\usepackage{float}

\newtheorem{thm}{Theorem}
\newtheorem{Lem}{Lemma}

\begin{document}
\maketitle
\begin{abstract}
\noindent
We consider Diophantine quintuples $\{a, b, c, d, e\}$, sets of integers with $a<b<c<d<e$ the product of any two elements of which is one less than a perfect square. Triples of the first kind are sets $\{A, B, C\}$ with $C\geq B^{5}$. We show that there are no Diophantine quintuples $\{a, b, c, d, e\}$ such that $\{a, b, d\}$ is a triple of the first kind.
\end{abstract}

\section{Introduction}
Define a Diophantine $m$-tuple as a set of $m$ distinct positive integers $a_{1}< a_{2}< \cdots < a_{m}$ such that $a_{i} a_{j} +1$ is a perfect square for all $1\leq i<j\leq m$. For example, the set $\{1, 3, 8, 120\}$ is a Diophantine quadruple. 
Throughout the rest of this article we simply refer to $m$-tuples, and not to Diophantine $m$-tuples.

One may extend any triple $\{a, b, c\}$ to a quadruple $\{a, b, c, d_{+}\}$ where
\begin{equation*}\label{d+}
d_{+} = a + b+ c+ 2abc + 2rst, \quad
r= \sqrt{ab +1}, \quad s = \sqrt{ac +1}, \quad t = \sqrt{bc+1},
\end{equation*}
by appealing to a result by Arkin, Hoggatt and Straus \cite{AHS}. Such a quadruple $\{a, b, c, d_{+}\}$ is called a \textit{regular} quadruple. Arkin, Hoggatt and Straus conjectured that all quadruples are regular.
Note that any possible quintuple $\{a, b, c, d, e\}$ contains the quadruples $\{a, b, c, d\}$ and $\{a, b, c, e\}$. If the conjecture by Arkin, Hoggatt and Straus were true then $d_{+} = d = e$, whence $d$ and $e$ are not distinct. This implies that there are no quintuples. A partial result towards this is the following theorem, proved by Fujita \cite{fujita2009any}.

\begin{thm}[Fujita]\label{thm:fuj2}
If $\{a,b,c,d,e\}$ is a Diophantine quintuple with $a<b<c<d<e$, then $d=d_+$.
\end{thm}

When attempting to bound the number of possible quintuples $\{a, b, c, d, e\}$ it is useful to examine the relative size of $b$ and $d$. To this end, Fujita \cite{Fujita} considered three classes of triples $\{A, B, C\}$, namely, triples of the \textit{first kind} in which $C> B^{5}$; triples of the \textit{second kind} in which $B> 4A$ and $B^{2} \leq C \leq B^{5}$; and triples of the \textit{third kind} in which $B> 12 A$ and $B^{5/3} < C < B^{2}$.

Lemma 4.2 in \cite{EFF} states that every quadruple contains a triple of the first, second, or third kind. Specifically, we have
\begin{Lem}[Lemma 4.2 in \cite{EFF}]\label{ELem}
Let $\{a, b, c, d, e\}$ be a Diophantine quintuple with $a<b<c<d<e$. Then
\begin{enumerate}
\item
$\{a, b, d\}$ is a triple of the first kind, or
\item
\begin{enumerate}[(i)]
\item $\{a, b, d\}$  is of the second kind, with $4ab + a + b \leq c \leq b^{3/2}$, or
\item $\{a, b, d\}$  is of the second kind, with $c = a+ b + 2r$, or
\item $\{a, b, d\}$  is of the second kind, with $c> b^{3/2}$, or
\item $\{a, c, d\}$ is of the second kind, with $b<4a$ and $c = a + b + 2r$, or
\end{enumerate}
\item $\{a, c, d\}$ is of the third kind, with $b<4a$ and $c = (4ab + 2)(a + b - 2r) + 2(a+b)$.
\end{enumerate}
\end{Lem}
Lemma \ref{ELem} allows us to take aim at various triples: to prove there are no quintuples one need only prove the nonexistence of quintuples containing each of the kinds of triples listed in Lemma \ref{ELem}. The aim of this paper is to show that there are no quintuples satisfying part 1 of Lemma \ref{ELem}.

The current best estimate\footnote{The second author has announced \cite{TrudgianQ} that there are at most $2.1\cdot 10^{29}$ quintuples, and that there are no quintuples of the kind 3 in Lemma \ref{ELem}.}, by Elsholtz, Filipin and Fujita \cite{EFF}, is that there are at most $6.8\cdot 10^{32}$ quintuples comprising
\begin{itemize}
\item $5.05\cdot 10^{15}$ possible quintuples derived from triples of the first kind,
\item $6.72\cdot 10^{32}$ possible quintuples derived from triples of the second kind, and
\item $1.92\cdot 10^{26}$ possible quintuples derived from triples of the third kind.
\end{itemize}

Attempts to bound the total number of quintuples have used a result by Matveev \cite{Matveev} on linear forms of logarithms. This is not required for triples of the first kind, which greatly simplifies the exposition. Rather, the bounds on $b$ and $d$ come from work of Filipin and Fujita \cite{FilFuj} wherein so-called gap principles between solutions of Pell's equation are considered.

In \S \ref{s2} we improve Fujita and Filipin's proof slightly. This enables us to perform computations in \S \ref{s4} that prove:

\begin{thm}\label{Main}
There are no Diophantine quintuples $\{a, b, c, d, e\}$ such that $\{a, b, d\}$ is a Diophantine triple with $d>b^{5}$.
\end{thm}
Independently of this work, Cipu \cite{Cipu} has considered the same problem. Indeed, his Theorem 1.1 states that ``the quadruple left after removing the largest entry of a Diophantine quintuple contains no standard triple of the first kind''. This is somewhat stronger than our Theorem \ref{Main}, though either result may be used to eliminate the triples in part 1 of Lemma \ref{ELem}. Moreover, Cipu shows \cite[Thm.\ 1.1]{Cipu} that there are at most $10^{31}$ quintuples.

 We were only made aware of this result when this paper was essentially completed. Elements of Cipu's paper, such as the reliance on inequalities of the form (\ref{hammond}) and (\ref{lara}) are similar to ours. Our computational approach is different, and, in particular, our Algorithm \ref{alg:2} introduces some new ideas to the field.

\subsection*{Acknowledgements}
We are grateful to Mihai Cipu, Andrej Dujella, Christian Elsholtz, Alan Filipin, Yasutsugu Fujita, Alain Togb\'{e} for useful discussions and for making us aware of the work by Cipu. 

\section{Bounds considered by Filipin and Fujita}\label{s2}
If a double or a triple cannot be extended to form a quintuple, we need not consider them in what follows. We call such doubles and triples \textit{discards}. Many discards are known; we require only a few. The first, due to Fujita \cite{Fujitak} (see also \cite{Bug}), that shows that the double $\{k, k+2\}$ is a discard. In addition, Filipin, Fujita and Togb\'{e} \cite[Cor.\ 1.6, 1.9]{FFT} proved that the following are discards for $k\geq 1$
\begin{equation}\label{neck1}
\begin{split}
&\{3k^{2} - 2k, 3k^{2} + 4k + 1\}, \quad\quad \{2(k+1)^{2} - 2(k+1), 2(k+1)^{2} + 2(k+1)\}, \\
& \{(k + 1)^2 - 1, (k + 1)^2 + 2 (k + 1)\}, \quad\quad \{k, 4k +4\}.
\end{split}
\end{equation}

For the following arguments we make frequent reference to Lemma 2.4 and (2.9) in \cite{FilFuj}. These rely on a property denoted as `Assumption 2.2' concerning the relations between the indexed solutions of the Pellian equations associated with the hypothetical quintuple. For ease of exposition, we do not write out the details of Assumption 2.2; we merely note that, for its application to quintuples containing triples of the kind in 1 in Lemma \ref{ELem}, it is satisfied as per \cite[p.\ 303]{FilFuj}.   

 We also note, for the convenience of the reader, that $a'$ in \cite[(2.9)]{FilFuj} is defined as $a' = \max\{b-a, a\}.$ We assume that $\{a, b, d\}$ is a triple with $d> b^{5}$. We consider two cases according as $b\geq 2a$ or $b<2a$.
\subsection{When $b\geq 2a$}\label{first}
Combining (2.9) and Lemma 2.4(i) in \cite{FilFuj} we have
\begin{equation}\label{hammond}
\frac{0.178 a^{1/2} d^{1/2}}{4b} < \frac{\log (4.001 a^{1/2} (b-a)^{1/2} b^{2} d) \log(1.299 a^{1/2} b^{1/2} (b-a)^{-1} d)}{\log (4bd) \log(0.1053ab^{-1} (b-a)^{-3} d)}.
\end{equation}
It is easy to see that both factors of the numerator and the second factor of the denominator are increasing in $a$ for $b\geq 2a$. Hence, given that $1\leq a \leq b/2$, the right side of (\ref{hammond}) is bounded above by
\begin{equation}\label{dravid}
\frac{\log (2.0005b^{3} d) \log(1.8371d)}{\log (4bd) \log(0.1053b^{-4} d)}.
\end{equation}
It is easy to see that $(\log A_{1} x)(\log A_{2} x)/((\log A_{3} x)(\log A_{4} x))$ is decreasing in $x$ if $A_{1}> A_{3}$ and $A_{2}> A_{4}$. Since $b\geq 2$ it follows that (\ref{dravid}) is decreasing in $d$, and since $d> b^{5}$, we have from (\ref{hammond}) that 
\begin{equation}\label{pollock}
b^{3/2} < \frac{4}{0.178} \frac{\log (2.0005b^{8}) \log(1.8371b^{5})}{\log (4b^{6}) \log(0.1053b)}.
\end{equation}
We find that (\ref{pollock}) holds provided that $b\leq 50$; in \cite{FilFuj} the bound derived is $b\leq 52$. Thus we need only consider those pairs $\{a, b\}$ for which $1 \leq a \leq b/2 \leq 25$. We enumerate these pairs, and test them against the inequality in (\ref{hammond}). After discarding pairs containing $\{k, k+2\}$ and those doubles in (\ref{neck1}) we find that the only possibilities are
\begin{equation}\label{stumped}
\{1, 15\}, \quad \{1, 24\},\quad \{1, 35\},\quad \{2, 24\},\quad  \{3, 21\}.
\end{equation}
For each of these five doubles we insert values of $a, b$ into (\ref{hammond}), and solve for $d$. For example, with $\{1, 15\}$ we find that $d\leq 5.2\cdot 10^{6}$. We now search for all those $d$ with $15^{5}< d \leq 5.2\cdot 10^{6}$ such that $\{1, 15, d\}$ is a triple. We find only one such value of $d$, namely $d = 2030624$. We now search for all those $c\in(15, 2030624)$ such that $\{1, 15, c, 2030624\}$ is a quadruple. This yields exactly one value of $c$, namely $c = 32760$. Thus, one possible quadruple is $\{1, 15, 37260, 2030624\}$. We continue in this way with each of the doubles in (\ref{stumped}). We find that the only possible quadruples are
\begin{equation}\label{qantas}
\{1, 15, 37260, 2030624\}, \quad \{1, 24, 148995, 14600040\}, \quad \{1, 35, 494208, 70174128\}.
\end{equation}

Note that each of the quadruples in (\ref{qantas}) has $a=1$. The proof of Lemma 2.4(i) in \cite{FilFuj}, that leads to the `0.178' on the right side of (\ref{hammond}) can be improved significantly if it is known that $a=1$. Indeed, if one there assumes $n\leq 0.45 b^{-1} c^{1/2}$ one obtains (2.4) in \cite{FilFuj} as well as the desired contradiction, as before. Thus, for $a=1$ one may replace the `0.178' in (\ref{hammond}) by `0.45'. It now follows that the double $\{1, 15\}$ when extended to a triple $\{1, 15, d\}$ must have $d< 2\cdot 10^{6}$, whence we eliminate the first quadruple in (\ref{qantas}). The other two quadruples are similarly eliminated. We conclude that there are no quintuples with $b\geq 2a$ for which $\{a, b, d\}$ is a triple of the first kind.
\subsection{When $b< 2a$}\label{second}
In this case we combine (2.9) and Lemma 2.4(iii) in \cite{FilFuj} and obtain
\begin{equation}\label{lara}
\frac{a^{-1/2} d^{1/8}}{4} < \frac{\log (4.001ab^{2}d) \log(1.299a^{1/2} b^{1/2} (b-a)^{-1} d)}{\log (4bd) \log(0.1053d b^{-1} (b-a)^{-2})}.
\end{equation}
We need a lower bound on $b-a$: given Fujita's result that $\{k, k+2\}$ cannot be extended to a quintuple, we can write $b-a\geq 3$. Since, again, the logarithms dependent upon $a$ in (\ref{lara}) are increasing with $a$, and since $a+ 3 \leq b <2a$ we rewrite (\ref{lara}) as
\begin{equation}\label{turner}
b^{1/8} < 4 \frac{\log (4.001 b^{8}) \log(0.433 b^{6})}{\log(4b^{6})\log(0.4212 b^{2})}.
\end{equation}
We find that (\ref{turner}) is true provided that $b\leq 4.69\cdot 10^{9}$. 
In fact, we can squeeze a little more out of the argument. Let $\alpha\in[1/2, 1)$ be a parameter to be chosen later. Then, for $a \leq \alpha b$ we have
\begin{equation}\label{steve}
b^{1/8} < 4\alpha^{1/2} \frac{\log (4.001 \alpha b^{8}) \log(	1.299 \alpha^{1/2} b^{6}(1-\alpha)^{-1})}{\log(4b^{6})\log(0.4212b^{2})},
\end{equation}
whereas, for $a> \alpha b$ we have
\begin{equation}\label{mark}
b^{1/8} < 4 \frac{\log (4.001 b^{8}) \log(0.433 b^{5})}{\log(4b^{6})\log(0.1053 b^{2}(1-\alpha)^{-2})}.
\end{equation}
Therefore $b^{1/8}$ is less than the maximum of the right sides in (\ref{steve}) and (\ref{mark}). We find that choosing $\alpha= 0.9862$ gives $b \leq 1.3\cdot 10^{9}$. Filipin and Fujita \cite{FilFuj} proved that $b< 10^{10}$: while our improvement is only slight, it makes the problem computationally tractable.

Filipin and Fujita proved also that $d< b^{9}$; this was improved in \cite{EFF} to $d< b^{7.7}$. We use the weaker bound $d<b^{8}$ for computational convenience.

We hope now to search for possible quadruples $\{a, b, c, d\}$ with $\{a,b,d\}$ a triple of the first kind. We first enumerate all double $\{a, b\}$ with $a\geq 1$, $a+2<b<2a$ and $b \leq 1.3\cdot 10^{9}$. For each such doubles we enumerate all $c$ where $c<b^{8}$ such that $\{a, b, c\}$ is a triple. We now appeal to Theorem \ref{thm:fuj2} and compute $d=d_+$. If $b^5<d<b^8$ and $a$, $b$ and $d$ satisfy inequality (\ref{lara}), then we add the quadruples to our \textit{initial list}.

\subsection{Specific bounds for $d$}
There are a little under 11 million quadruples in our initial list (details are given in \S \ref{s4}). We now propose a criterion against which to test these specimens.

Each quadruple $\{a, b, c, d\}$ gives rise to a system of Pellian equations the solutions to which are indexed by integers $m$ and $n$ --- see \cite[pp.\ 294-295]{FilFuj}. One obtains tighter bounds by showing that $m$ and $n$ must be of roughly the same size. Implicit in the proof of Lemma 2.3 in \cite{FilFuj} is the following problem. Given $\{a, b, d\}$ a triple of the first kind, and given $m, n$ with $m\geq 3$ and $n\geq 2$ we want to find good bounds on
\begin{equation}\label{boycott}
\frac{m}{n} < \frac{\log(2.001^{2} b d) + \frac{1}{n} \log(1.994 a^{1/2} b^{-1/2})}{\log(1.994^{2} a d)} = w(a, b, d, n),
\end{equation}
say.
Note that  $w$ is decreasing with $n$. Let $n_{0}$ be the smallest value of $n$ such that $(n_{0}+1)/n_{0} < w(a, b, d, n_{0}) = \gamma_{1}$, say.

Now let $\gamma_{2}$ be any number satisfying $\gamma_{2} > \gamma_{1}^{2}$, and let $\gamma_{3}$ be any number less than
\begin{equation}\label{ranji}
\frac{1}{\gamma_{1}} \sqrt{1 + \frac{1}{ad}} \left(\sqrt{\frac{\gamma_{2} ad +1}{ad +1}} - \gamma_{1}\right)=v(a,d,\gamma_1,\gamma_2).
\end{equation}
Filipin and Fujita have $\gamma_{1} = 1.2, \gamma_{2} = 1.45, \gamma_{3} = 0.0033$. This gives them a criterion against which to test quadruples provided that $b\geq 1.45 a$. We should like to take $\gamma_{2}$ to be less than 1.45, so that we can test more quadruples. We do this in the following lemma.

\begin{Lem}
Assume that $\gamma_{2} a \leq b < 2a$, and that $\{a, b, d\}$ is a triple with $d>b^{5}$. Then
\begin{equation}\label{russell}
\frac{\gamma_{3} a^{1/2} b^{-1} d^{1/2}}{4} < \frac{\log (4.001ab^{2}d) \log(1.299a^{1/2} b^{1/2} (b-a)^{-1} d)}{\log (4bd) \log(0.1053d b^{-1} (b-a)^{-2})}.
\end{equation}
\end{Lem}
\begin{proof}
Using (\ref{ranji}), and that fact that $\gamma_{2}> \gamma_{1}^{2}$, the proof follows exactly the same lines as in the proof of Lemma 2.4(ii) in \cite{FilFuj}. We complete the proof by combining the bound on $n$ with (2.9) in \cite{FilFuj}.
 \end{proof}

\section{Computations}\label{s4}
We first present a simple result on triples.
\begin{Lem}\label{lem:comp}

Let $a,b,r$ be positive integers with $a<b$ such that $ab+1=r^2$. That is, $\{a,b\}$ is a double. Then all admissible $c_n$ such that $\{a,b,c_n\}$ is a triple are of the form
\begin{equation*}
c_n=\frac{x_n^2-1}{a}=\frac{y_n^2-1}{b}
\end{equation*}
where $c_n$ is an integer and the $x_n,y_n$ are integer solutions to
\begin{equation}\label{eq:pell1}
bx^2-ay^2=b-a.
\end{equation}
\begin{proof}
For $c_n$ to be admissible, we have
\begin{equation*}
ac_n=x^2-1
\end{equation*}
and
\begin{equation*}
bc_n=y^2-1
\end{equation*}
for some integers $x,y$. We now simply eliminate $c_n$.
\end{proof}
\end{Lem}

We now require an efficient algorithm to identify low-lying solutions to (\ref{eq:pell1}). We start by dividing throughout by $g=(a,b)$ to get
\begin{equation*}
b^\dagger x^2-a^\dagger y^2=b^\dagger -a^\dagger
\end{equation*}
and then write $X=b^\dagger x$, $D=a^\dagger b^\dagger$ and $N=b^\dagger(b^\dagger -a^\dagger)$ to get
\begin{equation}\label{eq:Pell}
X^2-Dy^2=N.
\end{equation}
Since $D$ is not a square by construction, this is an example of Pell's equation which has been widely studied (see, for example, \cite{LevequeBook}). To find solutions, we first use Lagrange's PQa algorithm to find $(u,v)$, the fundamental solution to
\begin{equation*}\label{eq:Pellp}
X^2-Dy^2=1.
\end{equation*}

We then use brute force\footnote{Theorem 6.2.5 of \cite{Mollin1997} gives upper bounds on such a brute force approach.} to locate all the fundamental solutions to (\ref{eq:Pell}). Each such fundamental solution potentially gives rise to two infinite sequences of solutions, one generated by
\begin{equation*}
X_{n+1}+y_{n+1}\sqrt{D}=(X_n+y_n\sqrt{D})(u+v\sqrt{D})
\end{equation*}
and the other by
\begin{equation*}
X_{n+1}+y_{n+1}\sqrt{D}=(X_n+y_n\sqrt{D})(u-v\sqrt{D}).
\end{equation*}

Thus we have a series of recurrence relations, indexed by $i$ that will generate all possible solutions, which we denote
\begin{equation*}
x_{i,n+1}=fx_i(x_{i,n},y_{i,n}) \quad \textrm{and} \quad
y_{i,n+1}=fy_i(x_{i,n},y_{i,n})
\end{equation*}
for $i\in[0,I-1]$.

We can now proceed as described in Algorithm \ref{alg:1} to find all quadruples $\{a,b,c,d\}$ with $0<a<b<c<d$, $a+2<b<2a$, $b<1.3\cdot 10^9$ and $b^5<d<b^8$ that satisfy inequality (\ref{lara}).

\begin{algorithm}\label{alg:1}
\SetAlgoLined
\For{$r\leftarrow 2$ \KwTo $1.3\cdot 10^9$}{
   $w\leftarrow r^2-1$\;
   \For{$a|w$, $a<r-1$}{
      $b\leftarrow w/a$\;
      \lIf{$b\geq 2a$ {\bf or} $b>1.3\cdot 10^9$} {{\bf continue}}
      Solve Pell's equation $bx^2-ay^2=b-a$ to get $I$ base solutions\;
      \For{$i\in[0,I-1]$}{
         $x\leftarrow x_{i,0}$\;
         $y\leftarrow y_{i,0}$\;
         $c\leftarrow \frac{x^2-1}{a}$\;
         \While{$c\leq b$}{
            $t\leftarrow fx_i(x,y)$\;
            $y\leftarrow fy_i(x,y)$\;
            $x\leftarrow t$\;
            $c\leftarrow \frac{x^2-1}{a}$\;}
         $d\leftarrow a+b+c+2abc+2r\sqrt{ac+1}\sqrt{bc+1}$\;
         \While{$d\leq b^5$}{
            $t\leftarrow fx_i(x,y)$\;
            $y\leftarrow fy_i(x,y)$\;
            $x\leftarrow t$\;
            $c\leftarrow \frac{x^2-1}{a}$\;
            $d\leftarrow a+b+c+2abc+2r\sqrt{ac+1}\sqrt{bc+1}$\;}
         \While{$d<b^8$}{
            \If{$c\in\mathbb{Z}$ {\bf and} $a,b,d$ {\rm satisfy inequality (\ref{lara})}}{{\bf output} $\{a,b,c,d\}$}
            $t\leftarrow fx_i(x,y)$\;
            $y\leftarrow fy_i(x,y)$\;
            $x\leftarrow t$\;
            $c\leftarrow \frac{x^2-1}{a}$\;
            $d\leftarrow a+b+c+2abc+2r\sqrt{ac+1}\sqrt{bc+1}$}}}}
\caption{Producing the initial list.}
\end{algorithm}

We note that iterating over the divisors of $r^2-1$ is more efficient than iterating over $a$ and $b$ and testing whether $ab+1$ is a perfect square. When factoring $r^2-1$ for this purpose, we first factor $r\pm 1$ and merge the results.
 
We implemented Algorithm \ref{alg:1} in Pari \cite{Pari} (for the factoring) and in `C' using GMP \cite{GMP2013} (for everything else). We ran it on $5$ nodes of the University of Bristol Bluecrystal Phase III cluster \cite{ACRC2014} each of which comprises two $8$ core Intel Xeon E5-2670 CPUs running at $2.60$ GHz. The total time used across these nodes was about $40$ hours.

There are $4,038,480,906$ pairs $\{a,b\}$ with $ab+1$ a perfect square, $a+2<b<2a$ and $b\leq 1.3\cdot 10^9$. From these pairs, we obtained $12,115,454,363$ potential quadruples with $b<c<d$ and $b^5<d<b^8$. Applying inequality (\ref{lara}) eliminated all but $10,811,817$ of these quadruples and these survivors formed our initial list.

We now apply Algorithm \ref{alg:2}, based on inequalities (\ref{boycott}), (\ref{ranji}) and (\ref{russell}), to prune our initial list. To illustrate, consider the quadruple $\{a,b,c,d\}=\{8,15,21736,10476753\}$ which survives Algorithm \ref{alg:1}. We have 
\begin{equation*}
\frac{31}{30}>w(a,b,d,30)
\end{equation*}
but
\begin{equation*}
\frac{32}{31}<w(a,b,d,31)=1.0330\ldots=\gamma_1.
\end{equation*}
Now $b/a=15/8>\gamma_1^2$ so we will take $\gamma_2=15/8$. We now set
\begin{equation*}
\gamma_3=v(a,d,\gamma_1,\gamma_2)=0.32555\ldots
\end{equation*}
and we find that the left hand side of (\ref{russell}) is $49.67\ldots$ and the right hand side is $2.852\ldots$ so the inequality fails and we have eliminated this quadruple.

When coded in Pari, Algorithm \ref{alg:2} takes less than $20$ minutes on a single core to determine that none of the quadruples in our initial list is admissible. This proves Theorem \ref{Main}.

\begin{algorithm}\label{alg:2}
\SetAlgoLined
\For{$\{a,b,c,d\}\in$ \rm{initial list}}{
   $n_0\leftarrow$ smallest $n$ such that $\frac{n+1}{n}<w(a,b,d,n)$\;
   $\gamma_1\leftarrow w(a,b,d,n_0)$\;
   $\gamma_2\leftarrow\max\left(\frac{b}{a},\gamma_1^2\right)$\;
   $\gamma_3\leftarrow v(a,d,\gamma_1,\gamma_2)$\;
 \lIf{$a,b,d,\gamma_3$ {\rm satisfy inequality (\ref{russell})}}{{{\bf output} $\{a,b,c,d\}$}}}
\caption{Pruning the initial list.}
\end{algorithm}

\end{document}